\documentclass{mfat}
\pagespan{387}{392}

\usepackage{mmap}
\usepackage[utf8]{inputenc}

\newtheorem{theorem}{Theorem}[section]
\newtheorem{proposition}[theorem]{Proposition}
\newtheorem{corollary}[theorem]{Corollary}
\newtheorem{lemma}[theorem]{Lemma}
\newtheorem{remark}{Remark}

\newtheorem{definition}[theorem]{Definition}

\begin{document}

\title[L-Dunford-Pettis property in Banach spaces]
    {L-Dunford-Pettis property in Banach spaces}

\author{A.~Retbi}
\address{Universit\'{e} Ibn Tofail, Facult\'{e} des Sciences, D\'{e}%
partement de Math\'{e}matiques, B.P. 133, K\'{e}nitra, Morocco}
\email{abderrahmanretbi@hotmail.com}

\author{B.~El~Wahbi}

\date{01/03/2016;\ \ Revised 15/04/2016}

\begin{abstract}
  In this paper, we introduce and study the concept of
  L-Dunford-Pettis sets and L-Dunford-Pettis property in Banach
  spaces. Next, we give a characteriza\-tion of the L-Dunford-Pettis
  property with respect to some well-known geometric pro\-per\-ties of
  Banach spaces. Finally, some complementability of operators on
  Banach spaces with the L-Dunford-Pettis property are also
  investigated.
\end{abstract}

\maketitle

\section{Introduction and notation}
A norm bounded  subset $A$ of a Banach space $X$ is called
Dunford-Pettis (DP for short) if every weakly null sequence
$(f_{n})$ in $X^{\prime}$ converge uniformly to zero on $A$, that
is, $\lim_{n\rightarrow \infty }\sup_{x\in A}
\left|f_{n}(x)\right|=0$.

An operator $T$ between two Banach spaces $X$ and $Y$ is completely
continuous if T maps weakly null sequences into norm null ones.

Recall from \cite{Y}, that an operator $T:X\rightarrow Y$ between two
Banach spaces is Dunford-Pettis completely continuous (abb. $DPcc$) if
it carries a weakly null sequence, which is a DP set in $X$ to norm
null ones in $Y$. It is clear that every completely continuous
operator is DPcc. Also every weakly compact operator is DPcc (see
Corollary 1.1 of \cite{Y}).

A Banach space $X$ has:
\begin{itemize}
\item [--] a relatively compact Dunford-Pettis property (DPrcP for
  short) if every Dunford-Pettis set in $X$ is relatively compact
  \cite{E}. For example, every Schur spaces have the DPrcP.
\item [--] a Grothendieck property (or a Banach space $X$ is a
  Grothendieck space) if weak$^{\star}$ and weak convergence of
  sequences in $X^{\prime}$ coincide. For example, each reflexive
  space is a Grothendieck space.
\item [--] a Dunford-Pettis property (DP property for short) if every
  weakly compact operator $T$ from $X$ into another Banach space $Y$
  is completely continuous, equivalently, if every relatively weakly
  compact subset of $X$ is DP.
\item [--] a reciprocal Dunford-Pettis property (RDP property for
  short) if every completely continuous operator on $X$ is weakly
  compact.
\end{itemize}

A subspace $X_{1}$ of a Banach space $X$ is complemented if there
exists a projection $P$ from $X$ to $X_{1}$ (see page 9 of
\cite{B}).

Recall from \cite{A}, that a Banach lattice is a Banach space
$(E,\Vert \cdot \Vert )$ such that $E$ is a vector lattice and its
norm satisfies the following property: for each $x,y\in E$ such that
$|x|\leq |y|$, we have $\Vert x\Vert \leq \Vert y\Vert $.

We denote by $c_{0}$, $\ell^{1}$, and $\ell^{\infty}$ the Banach
spaces of all sequences converging to zero, all absolutely summable
sequences, and all bounded sequences, respectively.

Let us recall that a norm bounded subset $A$ of a Banach space
$X^{\prime}$ is called L-set if every weakly null sequence $(x_{n})$
in $X$ converge uniformly to zero on $A$, that is,
$\lim_{n\rightarrow \infty }\sup_{f\in A}
\left|f(x_{n})\right|=0$. Note also that a Banach space $X$ has the
RDP property if and only if every L-set in $X^{\prime}$ is relatively
weakly compact.

In his paper, G. Emmanuelle in \cite{EM} used the concept of L-set to
characterize Banach spaces not containing $\ell^{1}$, and gave several
consequences concerning Dunford-Pettis sets. Later, the idea of L-set
is also used to establish a dual characterization of the Dunford-Pettis
property \cite{G}.

The aim of this paper is to introduce and study the notion of
L-Dunford-Pettis set in a Banach space, which is related to the
Dunford-Pettis set (Definition \ref{D}), and note that every L-set in
a topological dual of a Banach space is L-Dunford-Pettis set
(Proposition~\ref{prop}). After that, we introduce the
L-Dunford-Pettis property in Banach space which is shared by those
Banach spaces whose L-Dunford-Pettis subsets of his topological dual
are relatively weakly compact (Definition \ref{De}). Next, we obtain
some important consequences. More precisely, a characterizations of
L-Dunford-Pettis property in Banach spaces in terms of DPcc and weakly
compact operators (Theorem \ref{th2}), the relation between
L-Dunford-Pettis property with DP and Grothendieck properties (Theorem
\ref{cr}), a new characterizations of Banach space with DPrcP (resp,
reflexive Banach space) (Theorem \ref{th1}) (resp, Corollary
\ref{cor}).  Finally, we investigate the complementability of the
class of weakly compact operators from $X$ into $\ell^{\infty}$ in the
class of DPcc from $X$ into $\ell^{\infty}$ (Theorem \ref{th3} and
Corollary \ref{cor1}).

The notations and terminologies are standard. We use the symbols $X$,
$Y$ for arbitrary Banach spaces.  We denoted the closed unit ball of
$X$ by $B_{X}$, the topological dual of $X$ by $X^{\prime}$ and
$T^{\prime}: Y^{\prime}\rightarrow X^{\prime}$ refers to the adjoint
of a bounded linear operator $T: X\rightarrow Y$.  We refer the reader
for undefined terminologies to the references \cite{A, Meg, M}.

\section{Main results}

\begin{definition}\label{D} {\rm{Let $X$ be a Banach space. A norm
      bounded subset $A$ of the dual space $X^{\prime}$ is called an
      L-Dunford-Pettis set, if every weakly null sequence $(x_{n})$,
      which is a DP set in $X$ converges uniformly to zero on $A$, that
      is, $\lim_{n\rightarrow \infty }\sup_{f\in A}
      \left|f(x_{n})\right|=0$}}.
\end{definition}


For a proof of the next Proposition, we need the following Lemma which
is just Lemma~1.3 of \cite{Y}.
\begin{lemma}\label{le}
A sequence $(x_{n})$ in $X$ is DP if and only if $f_{n}(x_{n})\rightarrow 0$ as $n\to\infty$ for every weakly null sequence $(f_{n})$ in $X^{\prime}$.
\end{lemma}

The following Proposition gives some additional properties of
L-Dunford-Pettis sets in a topological dual Banach space.

\begin{proposition}\label{prop}
Let $X$ be a Banach space. Then
\begin{enumerate}
\item every subset of an L-Dunford-Pettis set in $X^{\prime}$ is L-Dunford-Pettis,
\item every L-set in $X^{\prime}$ is L-Dunford-Pettis,
\item relatively weakly compact subset of $X^{\prime}$ is L-Dunford-Pettis,
\item absolutely closed convex hull of an L-Dunford-Pettis set in $X^{\prime}$ is L-Dunford-Pettis.
\end{enumerate}
\end{proposition}

\begin{proof}
$(1)$ and $(2)$ are obvious.

$(3)$ Suppose $A\subset X^{\prime}$ is relatively weakly compact but it is not an L-Dunford-Pettis set. Then, there exists a weakly null sequence $(x_{n})$, which is a DP set in $X$, a sequence $(f_{n})$ in $A$ and an $\epsilon>0$ such that $\left|f_{n}(x_{n})\right|>\epsilon$ for all integer n. As $A$ is relatively weakly compact, there exists a subsequence $(g_{n})$ of $(f_{n})$ that converges weakly to an element $g$ in $X^{\prime}$. But from
\begin{center}
$\left|g_{n}(x_{n})\right|\leq\left|(g_{n}-g)(x_{n})\right|+
\left|g(x_{n})\right|$
\end{center}
and Lemma \ref{le}, we obtain that $\left|g_{n}(x_{n})\right|\rightarrow 0$ as $n\to\infty$. This is a contradiction.

$(4)$ Let $A$ be a L-Dunford-Pettis set in $X^{\prime}$, and $(x_{n})$ be a weakly null sequence, which is a DP set in $X$. Since
\begin{center}
$\sup_{f\in a\overline{co}(A)}\left|f(x_{n})\right|= \sup_{f\in
A}\left|f(x_{n})\right|$
\end{center}
for each n, where $a\overline{co}(A)= \overline{\left\{\sum^{n}_{i= 1}\lambda_{i}x_{i}: x_{i}\in A, \forall i,
\sum^{n}_{i=1}\left|\lambda_{i}\right|\leq1\right\}}$ is the absolutely closed convex hull of $A$
(see \cite[pp.~148, 151]{A}), then it is clear that $a\overline{co}(A)$ is L-Dunford-Pettis set in $X^{\prime}$.
\end{proof}
We need the following Lemma which is just Lemma 1.2 of \cite{Y}.
\begin{lemma}\label{lm}
  A Banach space X has the DPrcP if and only if any weakly null
  sequence, which is a DP set in $X$ is norm null.
\end{lemma}
From Lemma \ref{lm}, we obtain the following characterization of DPrcP
in a Banach space in terms of an L-Dunford-Pettis set of his
topological dual.
\begin{theorem}\label{th1}
A Banach space $X$ has the DPrcP if and only if every bounded subset of $X^{\prime}$ is an L-Dunford-Pettis set.
\end{theorem}
\begin{proof}
$(\Leftarrow)$ Let $(x_{n})$ be a weakly null sequence, which is a DP set in $X$. As
\begin{center}
$\left\|x_{n}\right\|= \sup_{f\in
B_{X^{\prime}}}\left|f(x_{n})\right|$
\end{center}
for each $n$, and by our hypothesis, we see that $\left\|x_{n}\right\|\rightarrow 0$ as $n\to \infty$. By Lemma \ref{lm} we deduce that $X$ has the DPrcP.

$(\Rightarrow)$ Assume by way of contradiction that there exist a
bounded subset $A$, which is not an L-Dunford-Pettis set of
$X^{\prime}$. Then, there exists a weakly null sequence $(x_{n})$,
which is a Dunford-Pettis set of $X$ such that $\sup_{f\in
  A}\left|f(x_{n})\right|>\epsilon>0$ for some $\epsilon>0$ and each
$n$. Hence, for every $n$ there exists some $f_{n}$ in $A$ such that
$\left|f_{n}(x_{n})\right|>\epsilon$.

On the other hand, since $(f_{n})\subset A$, there exist some $M>0$
such that $\left\|f_{n}\right\|_{X^{\prime}}\leq M$ for all $n$. Thus,
\begin{center}
$\left|f_{n}(x_{n})\right|\leq M\left\|x_{n}\right\|$
\end{center}
for each $n$, then by our hypothesis and Lemma \ref{lm}, we have $\left|f_{n}(x_{n})\right|\to 0$ as $n\to \infty$, which is impossible. This completes the proof.
\end{proof}

\begin{remark} {\rm{Note by Proposition \ref{prop} assertion (3) that
      every relatively weakly compact subset of a topological dual
      Banach space is L-Dunford-Pettis. The converse is not true in
      general. In fact, the closed unit ball $B_{\ell^{\infty}}$ of
      ${\ell^{\infty}}$ is L-Dunford-Pettis set (see
      Theorem~\ref{th1}), but it is not relatively weakly compact}}.
\end{remark}

We make the following definition.

\begin{definition}\label{De}
{\rm{A Banach space X has the L-Dunford-Pettis property, if every L-Dunford-Pettis set in $X^{\prime}$ is
relatively weakly compact}}.
\end{definition}
As is known a DPcc operator is not weakly compact in general. For example, the identity operator $Id_{\ell^{1}}: {\ell^{1}}\rightarrow {\ell^{1}}$ is DPcc, but it is not weakly compact.

In the following Theorem, we give a characterizations of L-Dunford-Pettis property of Banach space in terms of DPcc and  weakly compact operators.
\begin{theorem}\label{th2}
Let X be a Banach space, then the following assertions are equivalent:

\begin{enumerate}
\item X has the L-Dunford-Pettis property,
\item for each Banach space Y, every DPcc operator from X into Y is weakly compact,
\item every DPcc operator from X into $\ell^{\infty}$ is weakly compact.
\end{enumerate}

\end{theorem}

\begin{proof}
$(1)\Rightarrow(2)$ Suppose that $X$ has the L-Dunford-Pettis property and $T: X\rightarrow Y$ is DPcc operator. Thus $T^{\prime}(B_{Y^{\prime}})$
is an L-Dunford-Pettis set in $X^{\prime}$. So by hypothesis, it is relatively weakly compact and $T$ is a weakly compact operator.

$(2)\Rightarrow(3)$ Obvious.

$(3)\Rightarrow(1)$ If $X$ does not have the L-Dunford-Pettis property, there exists an L-Dunford-Pettis subset $A$ of $X^{\prime}$ that is not relatively weakly compact. So there is a sequence $(f_{n})\subseteq A$ with no weakly convergent subsequence. Now, we show that the operator
$T: X\rightarrow\ell^{\infty}$ defined by $T(x)= (f_{n}(x))$ for all $x\in X$ is DPcc but it is not weakly compact. As $(f_{n})\subseteq A$ is L-Dunford-Pettis set, for every weakly null sequence $(x_{m})$, which is a DP set in $X$ we have
\begin{center}
$\left\|T(x_{m})\right\|= \sup_{n}\left|f_{n}(x_{m})\right|\rightarrow0$, as $m\rightarrow\infty$,
\end{center}
so $T$ is a Dunford-Pettis completely continuous operator. We have $T^{\prime}((\lambda_{n})^{\infty}_{n=1})= \sum^{\infty}_{n=1}\lambda_{n}f_{n}$ for every $(\lambda_{n})^{\infty}_{n=1}\in\ell^{1}\subset(\ell^{\infty})^{\prime}$. If $e^{\prime}_{n}$ is the usual basis element in $\ell^{1}$ then  $T^{\prime}(e_{n}^{\prime})= f_{n}$, for all $n\in N$. Thus, $T^{\prime}$ is not a weakly compact operator and neither is $T$. This finishes the proof.
\end{proof}

\begin{theorem}\label{cr}
Let $E$ be a Banach lattice.\\
If $E$ has both properties of DP and Grothendieck, then it has the L-Dunford-Pettis pro\-perty.
\end{theorem}

\begin{proof}
Suppose that $T: E\rightarrow Y$ is DPcc operator. As $E$ has the DP property, it follows from Theorem 1.5 \cite{Y} that $T$ is completely continuous.

On the other hand, $\ell^{1}$ is not a Grothendieck space and
Grothendieck property is carried by complemented subspaces. Hence the
Grothendieck space $E$ does not have any complemented copy of
$\ell^{1}$. By \cite{N}, $E$ has the RDP property and so the
completely continuous operator $T$ is weakly compact. From Theorem
\ref{th2} we deduce that $E$ has the L-Dunford-Pettis property.
\end{proof}

\begin{remark} {\rm{Since $\ell^{\infty}$ has the Grothendieck and DP
      properties, it has the L-Dunford-Pettis property}}.
\end{remark}

Let us recall that $K$ is an infinite compact Hausdorff space if it is
a compact Hausdorff space, which contains infinitely many points.

For an infinite compact Hausdorff space $K$, we have the following
result for the Banach space $C(K)$ of all continuous functions on $K$
with supremum norm.

\begin{corollary}
If $C(K)$ contains no complemented copy of $c_{0}$, then it has L-Dunford-Pettis property.
\end{corollary}

\begin{proof}
  Since $C(K)$ contains no complemented copy of $c_{0}$, it is a
  Grothendieck space \cite{C}. On the other hand, $C(K)$ be a Banach
  lattice with the DP property, and by Theorem \ref{cr} we deduce that
  $C(K)$ has L-Dunford-Pettis property.
\end{proof}

\begin{corollary}\label{cor}
A DPrc space has the L-Dunford-Pettis property if and only if it is reflexive.
\end{corollary}

\begin{proof}
$(\Rightarrow)$ If a Banach space $X$ has the DPrcP, then by Theorem 1.3 of \cite{Y}, the identity operator $Id_{X}$ on $X$ is DPcc. As $X$ has the L-Dunford-Pettis property, it follows from Theorem \ref{th2} that $Id_{X}$ is weakly compact, and hence $X$ is reflexive.

$(\Leftarrow)$ Obvious.
\end{proof}
\begin{remark}
{\rm{Note that the Banach space $\ell^{1}$ is not reflexive and has the DPrcP, then from Corollary \ref{cor},
we conclude that $\ell^{1}$ does not have the L-Dunford-Pettis property}}.
\end{remark}

\begin{theorem}
If a Banach space X has the L-Dunford-Pettis property, then every complemented subspace of X has the L-Dunford-Pettis property.
\end{theorem}

\begin{proof}
Consider a complemented subspace $X_{1}$ of $X$ and a projection map $P: X\rightarrow X_{1}$. Suppose $T: X_{1}\rightarrow\ell^{\infty}$ is DPcc operator, then $TP:  X\rightarrow\ell^{\infty}$ is also DPcc. Since $X$ has L-Dunford-Pettis, by Theorem \ref{th2}, TP is weakly compact. Hence $T$
is weakly compact, also from Theorem \ref{th2} we conclude that $X_{1}$ has L-Dunford-Pettis, and this completes the proof.
\end{proof}
Let $X$ be a Banach space. We denote by $L(X, \ell^{\infty})$ the
class of all bounded linear operators from $X$ into $\ell^{\infty}$,
by $W(X, \ell^{\infty})$ the class of all weakly compact operators
from $X$ into $\ell^{\infty}$, and by $DPcc(X, \ell^{\infty})$ the
class of all Dunford-Pettis completely continuous operators from $X$
into $\ell^{\infty}$.

Recall that Bahreini in  \cite{B} investigated the
complementability of $W(X, \ell^{\infty})$ in $L(X, \ell^{\infty})$,
and she proved that if $X$ is not a reflexive Banach space, then $W(X,
\ell^{\infty})$ is not complemented in $L(X, \ell^{\infty})$. In the
next Theorem, we establish the complementability of $W(X,
\ell^{\infty})$ in $DPcc(X, \ell^{\infty})$.

We need the following lemma of \cite{K} .
\begin{lemma}\label{lem}
Let X be a separable Banach space, and $\phi: \ell^{\infty}\rightarrow L(X, \ell^{\infty})$ is a bounded linear operator with $\phi(e_{n})= 0$ for all n, where $e_{n}$ is the usual basis element in $c_{0}$. Then there is an infinite subset $M$ of $N$ such that for each $\alpha\in\ell^{\infty}(M)$, $\phi(\alpha)= 0$, where $\ell^{\infty}(M)$ is the set of all $\alpha= (\alpha_{n})\in\ell^{\infty}$ with $\alpha_{n}= 0$ for each $n\notin M$.
\end{lemma}

\begin{theorem}\label{th3}
If X does not have the L-Dunford-Pettis property, then $W(X, \ell^{\infty})$ is not complemented in $DPcc(X, \ell^{\infty})$.
\end{theorem}

\begin{proof}
Consider a subset $A$ of $X^{\prime}$ that is L-Dunford-Pettis but it is not relatively weakly compact. So
there is a sequence $(f_{n})$ in $A$ such that has no weakly convergent subsequence. Hence $S: X\rightarrow\ell^{\infty}$ defined by $S(x)= (f_{n}(x))$ is an DPcc operator but it is not weakly compact. Choose a bounded sequence $(x_{n})$ in $B_{X}$ such that $(S(x_{n}))$ has no weakly convergent subsequence. Let $X_{1}= \left\langle x_{n}\right\rangle$, the closed linear span of the sequence $(x_{n})$ in $X$. It follows that $X_{1}$ is a separable subspace of $X$ such that $S/X_{1}$ is not a weakly compact operator. If $g_{n}= f_{n}/X_{1}$, we have $(g_{n})\subseteq X_{1}^{\prime}$ is bounded and has no weakly convergent subsequence.

Now define the operator $T: \ell^{\infty}\rightarrow DPcc(X, \ell^{\infty})$ by $T(\alpha)(x)= (\alpha_{n}f_{n}(x))$, where $x\in X$ and $\alpha= (\alpha_{n})\in \ell^{\infty}$. Then
\begin{center}
$\left\|T(\alpha)(x)\right\|=\sup_{n}\left|\alpha_{n}f_{n}(x)\right|\leq \left\|\alpha\right\|.\left\|f_{n}\right\|.\left\|x\right\|<{\infty}$.
\end{center}
We claim that $T(\alpha)\in DPcc(X, \ell^{\infty})$ for each $\alpha=(\alpha_{n})\in\ell^{\infty}$.

Let $\alpha=(\alpha_{n})\in\ell^{\infty}$ and let $(x_{m})$ be a weakly null sequence, which is a DP set in $X$. As $(f_{n})$ is L-Dunford-Pettis set $\sup_{n}\left|f_{n}(x_{m})\right|\rightarrow 0$ as $m\rightarrow \infty$. So we have
\begin{center}
$\left\|T(\alpha)(x_{m})\right\|=\sup_{n}\left|\alpha_{n}f_{n}(x_{m})\right|\leq \left\|\alpha\right\|.\sup_{n}\left|f_{n}(x_{m})\right|\rightarrow 0$,
\end{center}
as $m\rightarrow \infty$. Then this finishes the proof that $T$ is a well-defined operator from $\ell^{\infty}$ into $DPcc(X, \ell^{\infty})$.

Let $R: DPcc(X, \ell^{\infty})\rightarrow DPcc(X_{1}, \ell^{\infty})$ be the restriction map and define
\begin{center}
$\phi: \ell^{\infty}\rightarrow DPcc(X_{1},\ell^{\infty})$\quad  {by} \quad  $\phi= RT$.

\end{center}
Now suppose that $W(X, \ell^{\infty})$ is complemented in $DPcc(X, \ell^{\infty})$ and
\begin{center}
$P: DPcc(X, \ell^{\infty})\rightarrow W(X, \ell^{\infty})$
\end{center}
is a projection. Define $\psi: \ell^{\infty}\rightarrow W(X_{1},\ell^{\infty})$ by $\psi= RPT$. Note that as $T(e_{n})$ is a one rank operator, we have $T(e_{n})\in W(X, \ell^{\infty})$. Hence
\begin{center}
$\psi(e_{n})= RPT(e_{n})= RT(e_{n})= \phi(e_{n})$
\end{center}
for all $n\in N$. From Lemma \ref{lem}, there is an infinite set $M\subseteq N$ such that $\psi(\alpha)=\phi(\alpha)$ for all $\alpha\in \ell^{\infty}(M)$. Thus $\phi(\chi_{M})$ is a weakly compact operator. On the other hand, if $e^{\prime}_{n}$ is the usual basis element of $\ell^{1}$, for each $x\in X_{1}$ and each $n\in M$, we have
\begin{center}
$(\phi(\chi_{M}))^{\prime}(e^{\prime}_{n})(x)= f_{n}(x)$.
\end{center}
Therefore $(\phi(\chi_{M}))^{\prime}(e^{\prime}_{n})= f_{n}/X_{1}= g_{n}$ for all $n\in M$. Thus $(\phi(\chi_{M}))^{\prime}$ is not a weakly compact
operator and neither is $\phi(\chi_{M})$. This contradiction ends the proof.
\end{proof}

As a consequence of Theorem \ref{th2} and Theorem \ref{th3}, we obtain the following result.

\begin{corollary}\label{cor1}
Let $X$ be a Banach space. Then the following assertions are equivalent:
\begin{enumerate}
\item X has the L-Dunford-Pettis property,
\item $W(X, \ell^{\infty})= DPcc(X, \ell^{\infty})$,
\item $W(X, \ell^{\infty})$ is complemented in $DPcc(X, \ell^{\infty})$.
\end{enumerate}
\end {corollary}


\end{document}